\documentclass[preprint]{elsarticle}

\usepackage{amsmath, amssymb, amsthm, amsfonts, mathtools, mathrsfs, mathdots}
\usepackage[shortlabels]{enumitem}
\usepackage{pifont}
\usepackage{url}

\newtheorem{thm}{Theorem}
\newtheorem{corollary}[thm]{Corollary}
\newtheorem{proposition}[thm]{Proposition}
\newtheorem{lemma}[thm]{Lemma}
\newtheorem{conjecture}{Conjecture}
\newdefinition{definition}{Definition}
\newtheorem{remark}{Remark}

\newtheorem{example}[thm]{Example}

\newcommand{\R}{\mathbb{R}}      
\newcommand{\K}{\mathbb{K}}      

\newcommand{\bS}{\mathbb{S}}      

\newcommand{\bK}{\mathbb{K}}

\newcommand{\ku}{\mathbf{u}}    
   
\newcommand{\kw}{\mathbf{w}}

\newcommand{\xmark}{\ding{55}}

\DeclareMathOperator{\Id}{\mathbb{I}}
\DeclarePairedDelimiterX{\norm}[1]{\lVert}{\rVert}{#1}  

\newcommand{\symn}[1]{\operatorname{sym}(#1)}
\newcommand{\skewn}[1]{\operatorname{skew}(#1)}

\DeclareMathOperator{\vect}{vec}
\DeclareMathOperator{\matt}{Mat}
\DeclareMathOperator{\svect}{s2vec}

\DeclarePairedDelimiter\floor{\lfloor}{\rfloor}
\newcommand\svec[1]{\svect (#1)}

\newcommand\ovec[1]{\vect (#1)}

\newcommand{\trace}{\operatorname{tr}}     
\DeclareMathOperator{\rank}{rank}
\newcommand{\Diag}{\operatorname{Diag}}     

\begin{document}
\begin{frontmatter}
\title{On the spectral structure of Jordan-Kronecker products of symmetric and skew-symmetric matrices\tnoteref{t1}}


\author[uw]{Nargiz Kalantarova\corref{cor1}}
\ead{n2kalant@uwaterloo.ca}
\author[uw]{Levent Tun\c{c}el}
\ead{ltuncel@uwaterloo.ca}
 
\cortext[cor1]{Corresponding author} 
\tnotetext[t1]{Some of the material in this thesis appeared in a preliminary form in Kalantarova's Ph.D. thesis \cite{thesis:nkalantarova}}
\address[uw]{Department of Combinatorics and Optimization, Faculty of Mathematics, University of Waterloo, 200 University Ave. W., Waterloo, ON, N2L 3G1, Canada.}

\begin{abstract}
Motivated by the conjectures formulated in 2003 \cite{TuncelWolk2005}, we study interlacing properties of the eigenvalues of $A\otimes B + B\otimes A$ for pairs of $n$-by-$n$ matrices $A, B$. We prove that for every pair of symmetric matrices (and skew-symmetric matrices) with one of them at most rank two, the \emph{odd spectrum} (those eigenvalues determined by skew-symmetric eigenvectors) of $A\otimes B + B\otimes A$ interlaces its \emph{even spectrum} (those eigenvalues determined by symmetric eigenvectors). Using this result, we also show that when $n \leq 3$, the odd spectrum of $A\otimes B + B\otimes A$ interlaces its even spectrum for every pair $A, B$. The interlacing results also specify the structure of the eigenvectors corresponding to the extreme eigenvalues. In addition, we identify where the conjecture(s) and some interlacing properties hold for a number of structured matrices. We settle the conjectures of \cite{TuncelWolk2005} and show they fail for some pairs of symmetric matrices $A, B$, when $n\geq 4$ and the ranks of $A$ and $B$ are at least $3$. 

\end{abstract}



\end{frontmatter}

\section{Introduction}
Kronecker products and interlacing are commonly used in matrix theory, spectral graph theory and in their applications. 
This work investigates interlacing properties of the eigenvalues of $A\otimes B + B\otimes A$ (for which the eigenvalues partition into two classes with respect to the structure of the eigenvectors they correspond to). These interlacing properties also specify the structure of the eigenvectors corresponding to the extreme eigenvalues which is crucial for determining the condition number of normal matrices and plays an important role in sensitivity analysis.

Given real numbers $a_1\geq a_2\geq \cdots \geq a_n$ and $b_1\geq b_2\geq \cdots \geq b_m$ where $m<n$, the sequence $\{b_i\}_{i=1}^m$ is said to \emph{interlace} $\{a_i\}_{i=1}^n$, if for every $i\in\{1,2,\ldots, m\}$,
\[
a_{i}\geq b_{i}\geq a_{i+n-m}.
\]
Interlacing has been studied extensively in the literature \cite{ThompsonFreede:1971, Thompson:1976}. It has many applications in matrix theory, perturbation theory, real stable polynomials and spectral graph theory \cite{Cantoni1976, Delsarte1984, Trench:1993, book:RahmanSchmeisser, DWagner, Haemers:95, Haemers:2014, Spielman:2013}, among others.

The {\it Kronecker product} of real matrices $A \in \R^{m\times n}$ and $B \in \R^{p \times q}$ is the $mp$-by-$nq$ matrix which is defined by
\begin{equation*}
A\otimes B := 
\begin{bmatrix} 
a_{11} B & a_{12} B & \cdots & a_{1n}B\\
a_{21} B & a_{22} B & \cdots & a_{2n}B\\
\vdots & \vdots & \ddots & \vdots \\
a_{m1} B & a_{m2} B & \cdots & a_{mn}B\\
\end{bmatrix}.
\end{equation*}

The Kronecker product of two matrices represents the tensor product of two special linear maps. It arises in signal processing, semidefinite programming, and quantum computing, see \cite{Pitsianis:1997, VanLoan2000} and the references therein. They have also been used extensively in the theory and applications of linear matrix equations such as Sylvester equations and algebraic Lyapunov equations \cite{VanLoan2000, Mehrmann:2000}, in some compressed sensing applications \cite{Vanderbei:2016} and in strengthening the semidefinite relaxations of second order cone constraints \cite{Anstreicher2017}. Low rank tensor approximations arising in the context of certain quantum chemistry systems \cite{VanLoanVokt2015}, linear systems and eigenvalue problems \cite{Kressner:2016} 
are among many other applications that utilize the rich properties of tensor products.

We adopt the following terminology. For a matrix $M\in\R^{m\times n}$, we denote the \emph{Frobenius norm} of $X$ by
$\norm{M}_F :=\left(\trace \left(M^\top M \right)\right)^{1/2}$, where $(\cdot)^\top$ represents the transpose and $\trace(\cdot)$ is the usual trace on square matrices. We denote the sets of {\it $n$-by-$n$ real symmetric} and \emph{$n$-by-$n$ real skew-symmetric matrices} by $\bS^n$ and $\bK^n$, respectively.
The dimensions of $\bS^n$ and $\bK^n$ are denoted by $\symn{n}:=n(n+1)/2$ and $\skewn{n}:=n(n-1)/2$, respectively.
If $X\in \R^{m\times n}$, $\ovec{X}$ is an $mn$-dimensional vector formed by stacking the columns of $X$ consecutively. The inverse of the $\ovec{\cdot}$ operation is called $\matt (\cdot)$, which matricizes the $mn$-dimensional vector to an $m$-by-$n$ matrix.
For matrices $A, B$ and $X$ of appropriate dimensions,
\begin{equation}
\label{ident:kronvec}
\small \ovec{AXB} = (B^\top\otimes A) \ovec{X}.
\end{equation}

\begin{definition}[Symmetric/skew-symmetric vector]
\label{defn:sym-skew-vec}
Let $P \in \R^{n^2\times n^2}$ be an involutory, symmetric matrix. We call $x \in \R^{n^2}$ a {\it symmetric vector} if $Px = x$, a {\it skew-symmetric vector} if $Px = -x$. 
\end{definition}

\noindent The \emph{commutation matrix} $T$ \cite{Neudecker:1979} takes the ``transpose" of $n^2$-dimensional vectors. This matrix is involutory and symmetric, i.e., it satisfies $T^2 = \Id$ and $T^\top = T$. Moreover, for every pair of $n$-by-$n$ matrices $A$ and $B$, $T(A\otimes B)T = B\otimes A$.

The \emph{symmetric Kronecker product} of $A, B \in \R^{n\times n}$ is defined as
\begin{equation}
\label{defn:symKron}
(A\overset{s}{\otimes} B) := \frac{1}{2}Q^\top(A\otimes B + B\otimes A)Q = Q^\top(A\otimes B)Q.
\end{equation}
See, for instance, \cite{TuncelWolk2005} for an alternative definition.
The {\it skew-symmetric Kronecker product} of $A, B\in \R^{n\times n}$ is defined as
\[
A\overset{\tilde{s}}{\otimes} B:= \frac{1}{2}\tilde{Q}^\top \left(A\otimes B+ B\otimes A\right)\tilde{Q} = \tilde{Q}^\top \left(A\otimes B\right)\tilde{Q}.
\]
Here, $Q\in \R^{n^2 \times \symn{n}}$ and $\tilde{Q}\in \R^{n^2 \times \skewn{n}}$ are such that their columns form an orthonormal basis for $n^2$-by-$1$ symmetric and skew-symmetric vectors, respectively. For explicit definitions of these matrices we refer the reader to \cite{thesis:nkalantarova}.
Note that $Q^\top Q = \Id$ and $\tilde{Q}^\top\tilde{Q} = \Id$. Both $Q Q^\top$ and $\tilde{Q} \tilde{Q}^\top$ are orthogonal projectors. The former is mapping every point in $\R^{n^2}$ onto the set of symmetric vectors in $\R^{n^2}$ \cite{TuncelWolk2005}, and the latter to the set of skew-symmetric vectors in $\R^{n^2}$. 

The spectral structure of $A\otimes B$ is well known \cite{book:HornJohnson}. 
However, the spectral structure of $A\otimes B + B\otimes A$, its relation to the eigenvalue/eigenvectors of the symmetric Kronecker product as well as the skew-symmetric Kronecker product have not been developed fully.

\begin{definition}[Jordan-Kronecker Product]
Let $A$, $B$ be $n$-by-$n$ real matrices. The \emph{Jordan-Kronecker product} of $A$ and $B$ is defined as
\[
A\otimes B + B\otimes A.
\]
\end{definition}
Indeed, this is the Jordan product of $A$ and $B$ \cite{book:Eves_matrix, Faraut:1994}
where the matrix multiplication is replaced by the Kronecker product. This is also related to the notion of \emph{Jordan triple product}, since
\[
(A\otimes B + B\otimes A)\ovec{X} = \ovec{AXB} + \ovec{BXA}, \;\;\; \forall A, B \in \bS^n.
\]

The Jordan-Kronecker product is a special form of a generalized continuous-time Lyapunov operator that has been studied extensively in the literature, see \cite{Stykel:2002} and the references therein.
A nice characterization for the eigenstructure of the Jordan-Kronecker product of $n\text{-by-}n$ matrices $A$ and $B$ is provided in \cite[Section 2]{TuncelWolk2005} which shows that the eigenvectors of 
$A\otimes B + B\otimes A$ can be chosen so that they form an orthonormal basis where each eigenvector is either a symmetric or a skew-symmetric vector, where $P$ in Definition~\ref{defn:sym-skew-vec} is the commutation matrix. 
In fact, one can observe that if the commutation matrix $T$ in
\[
A\otimes B + B\otimes A = A\otimes B + T(A\otimes B)T
\]
is replaced by an arbitrary symmetric, involutory {matrix\protect\footnote{We thank Chris Godsil for suggesting this generalization.}} then the above matrix has an eigenvector decomposition such that eigenvectors form an orthonormal basis where each is either a symmetric or a skew-symmetric vector.

The outline of our paper is as follows. In section 2, we review related literature on structured matrices with similar eigenvalue/eigenvector structure and interlacing properties. We contrast our work with the existing literature and summarize our main results in Table~\ref{table:summary}. We generalize the eigenstructure of the Jordan-Kronecker product. Then, in Section 2.1 we present affirmative results where interlacing properties hold and conjecture that the eigenvector corresponding to the smallest eigenvalue of the Jordan-Kronecker product of positive definite matrices is symmetric. In Section 2.2, we give a counterexample and settle the conjectures posed in \cite{TuncelWolk2005}. Also, we provide a family of matrices where the conjectures fail. In Section 3, we  discuss the eigenvalue/eigenvector structure of Lie-Kronecker products and conclude our paper.
\section{Main Results}
In \cite{TuncelWolk2005}, Tun\c{c}el and Wolkowicz conjectured interesting interlacing relations (see \cite[Conjecture 2.10-2.12]{TuncelWolk2005}) on the eigenvalues of the Jordan-Kronecker products of real symmetric matrices. These conjectures not only propose an ordering relation between certain eigenvalues of the Jordan-Kronecker products but also a symmetry structure for the eigenvectors corresponding to the extreme eigenvalues. In this paper, we investigate interlacing properties in general as well as the conjectures stated in \cite{TuncelWolk2005}. 

The characterization of structured matrices by a symmetry property of their eigenvectors is not a new concept. There are various structured matrices other than Jordan-Kronecker products such that they have an eigendecomposition where their eigenvector is either symmetric (if $Px=x$) or skew-symmetric (if $Px=-x$), where $P$ is a special involutory matrix. Here, the eigenvalue corresponding to a symmetric eigenvector is called \emph{even} and the one corresponding to a skew-symmetric eigenvector is called \emph{odd}. This terminology was also used for centrosymmetric matrices \cite{Andrew:1973} in which $P$ is the matrix with ones on the secondary diagonal (from the top right corner to the bottom left corner) and zeros elsewhere and also has been used with a general involutory matrix in \cite{Yasuda:2003}. Utilization of this structure leads to efficient solution for the eigenvalue problem (with $25$\% savings in computation time) and helps with exposing certain properties of the solution that are also shared by the approximate solution found by some numerical methods to solve certain Sturm-Liouville problems \cite{Andrew:1973}. 

Some sufficient conditions on the interlacing of odd and even eigenvalues are provided in \cite[Theorem 5, 6]{Cantoni1976} for centrosymmetric matrices, and in \cite{Trench:1993} for some real symmetric Toeplitz matrices. The interlacing property, alone, is algebraically interesting itself but it also plays an important role in answering the inverse eigenvalue problem \cite{Trench:1993, Landau:1994} for real symmetric Toeplitz matrices.
Another example of a set of matrices with such structured eigenvectors are \emph{perfect shuffle symmetric matrices} (e.g. matrices $A\in \bS^{n^2}$ such that 
$TAT = A$, see for instance \cite{Dangeli:2017}). Perfect shuffle matrices also arise when unfolding certain order 4-tensors in certain quantum chemistry applications \cite{VanLoanVokt2015}. We note that the Jordan-Kronecker product is a symmetric perfect shuffle matrix.

Such structures also appear in the singular vectors of certain Lyapunov operators. The sensitivity of the solution of the Lyapunov equations depends on the smallest singular values. When the smallest singular vector is symmetric this is advantageous for certain numerical algorithms \cite{ByersNash:1987}. Therefore, it is valuable to know whether these extreme singular vectors are symmetric or skew-symmetric. Some studies on the singular vector(s) corresponding to the smallest and largest singular values of a Lyapunov operator can be found in \cite{ByersNash:1987, ChenTian:2015}. For the \emph{generalized} continuous-time and discrete-time Lyapunov operators, the symmetry structure of the smallest and largest singular vectors of has been addressed in \cite{ChenTian:2016}. 

Let $s$ and $t$ denote the dimensions of the sets $\{x \in \R^{n^2}: Px = x \}$  and $\{x \in \R^{n^2}: Px = -x \}$, respectively, where $P$ is an $n^2$-by-$n^2$ symmetric, involutory matrix.  
Furthermore, let $\Theta\in \R^{n^2\times s}$ be a matrix such that $P\Theta = \Theta$ and let $\tilde{\Theta}\in \R^{n^2 \times t}$ be such that $P\tilde{\Theta}=-\tilde{\Theta}$. Then, the eigenvalue/eigenvector structure of the \emph{generalized Jordan-Kronecker product}, $ A\otimes B + P(A\otimes B)P$, is similar to that of Jordan-Kronecker product. The following is a generalization of Theorem 2.9 in \cite{TuncelWolk2005} that characterizes the eigenvector/eigenvalue structure of the generalized Jordan-Kronecker product.
\begin{proposition}
\label{thm:genJordan_eigenstructure}
Let $A, B\in \bS^n$ (or $\K^n$) and $P\in \R^{n^2\times n^2}$ be a symmetric, involutory matrix. Then, $A\otimes B + P(A\otimes B)P$ has an eigendecomposition such that the eigenvectors decompose into symmetric and skew-symmetric vectors. Furthermore, the eigenvalues corresponding to the symmetric eigenvectors are the eigenvalues of $2\Theta^\top(A\otimes B)\Theta$ and the ones corresponding to the skew-symmetric eigenvectors are the eigenvalues of $2\tilde{\Theta}^\top(A\otimes B)\tilde{\Theta}$, where $\Theta\in \R^{n^2\times s}$ and $\tilde{\Theta}\in \R^{n^2 \times t}$ such that $P\Theta = \Theta$ and $P\tilde{\Theta}=-\tilde{\Theta}$.
\end{proposition}
\begin{proof}
Let $C:=A\otimes B + P(A\otimes B)P$. Also, let
$\Theta':= \begin{bmatrix} \Theta & \tilde{\Theta} \end{bmatrix}$, where $\Theta$ is the $n^2$-by-$s$ matrix such that $P\Theta = \Theta$ and $\tilde{\Theta}$ is the $n^2$-by-$t$ matrix such that $P\tilde{\Theta}=-\tilde{\Theta}$. Then  
\begin{align*}
{\Theta'}^\top C\Theta' = 
\begin{bmatrix} 
\Theta^\top C \Theta & \Theta^\top C \tilde{\Theta}\\
\tilde{\Theta}^\top C \Theta & \tilde{\Theta}^\top C \tilde{\Theta}
\end{bmatrix}
= 
\begin{bmatrix} 
\Theta^\top C \Theta & 0\\
0 & \tilde{\Theta}^\top C \tilde{\Theta}
\end{bmatrix}.
\end{align*}
The off-diagonal blocks are zero since 
 $\Theta^\top C \tilde{\Theta} = \Theta^\top P C P \tilde{\Theta} =  -\Theta^\top C \tilde{\Theta}$.
Note that $\Theta^\top(A{\otimes} B)\Theta$ and $\tilde{\Theta}^\top(A{\otimes} B)\tilde{\Theta}$ are symmetric.
Let $\Theta^\top(A\otimes B)\Theta = U{\Lambda}_{e}U^\top$ and
$\tilde{\Theta}^\top(A{\otimes} B)\tilde{\Theta} = V {\Lambda}_{o}V^\top$
be the spectral decomposition of $\Theta^\top(A\otimes B)\Theta$ and $\tilde{\Theta}^\top(A{\otimes} B)\tilde{\Theta}$, respectively.
Then,
\begin{align*}
 C &= \begin{bmatrix} \Theta U & \tilde{\Theta} V \end{bmatrix} \begin{bmatrix} 2{\Lambda}_{e} & 0\\ 0 & 2{\Lambda}_{o} \end{bmatrix}
\begin{bmatrix} \Theta U & \tilde{\Theta} V \end{bmatrix}^\top.
\end{align*}
By the definition of $\Theta$ and $\tilde{\Theta}$, the columns of $\Theta U$ are symmetric and the columns of $\Theta V$ are skew-symmetric. Therefore, the even spectrum of $C$ consists of the eigenvalues of $2(\Theta^\top(A\otimes B)\Theta)$ and the odd spectrum of $C$ consists of the eigenvalues of 
$2(\tilde{\Theta}^\top(A{\otimes} B)\tilde{\Theta})$.
\end{proof}
We define some interlacing properties which will be referred throughout the paper in order to avoid repetition of the conjectures provided in \cite{TuncelWolk2005}.
\begin{definition}
Let $A, B \in \bS^n$ (or $\bK^n$) and let  $\lambda_1\geq \cdots \geq \lambda_s$ and $\beta_1\geq \cdots \geq \beta_t$ denote the even and the odd eigenvalues of $C:=A\otimes B+ B\otimes A$, respectively, in non-increasing order, where $s:=(n+1)n/2$ and $t:=(n-1)n/2$. 
\begin{itemize}
\item \textbf{Interlacing Property}: The \emph{odd eigenvalues of $C$ interlace its even eigenvalues} if for an eigenvalue $\beta_i$ in the odd spectrum of $C$, there are eigenvalues, $\lambda_{s-t+i}, \lambda_i$ in the even spectrum of $C$ such that
\[
\lambda_i \geq \beta_i \geq \lambda_{s-t+i}, \;\; \forall i\in\{1,\ldots, t\}.
\]
\item \textbf{Weak Interlacing Property}:
$A, B$ satisfies \emph{weak interlacing}, if
\begin{equation}
\label{conj1:eq1}
\lambda_s :=\min_{Tu = u}\dfrac{u^\top (A \otimes B) u}{u^\top u} \leq \min_{Tw = -w}\dfrac{w^\top (A \otimes B) w}{w^\top w}=: \beta_t,
\end{equation}
\begin{equation}
\label{conj1:eq2}
\beta_1:=\max_{Tw = -w}\dfrac{w^\top (A \otimes B) w}{w^\top w} \leq \max_{Tu = u}\dfrac{u^\top (A \otimes B) u}{u^\top u}=:\lambda_1
\end{equation}
or equivalently,
\begin{equation}
\label{conj1_2:eq1}
\min_{U\in \bS^n, \norm{U}_F=1}\trace (A UBU) \leq \min_{W\in \bK^n, \norm{W}_F=1}\trace \left(A WBW^\top\right), \end{equation}
\begin{equation}
\label{conj1_2:eq2}
\max_{U\in \bS^n, \norm{U}_F=1}\trace (A UBU) \geq \max_{W\in \bK^n, \norm{W}_F=1}\trace \left(A WBW^\top\right).
\end{equation}
\item \textbf{Strong Interlacing Property}:
$A, B$ satisfies \emph{strong interlacing} if for each $k$th eigenvalue of $C$ that is odd, where $k\in\{2,\ldots, n-1\}$, the $(k+1)$th and $(k-1)$th eigenvalues of $C$ are associated with symmetric eigenvectors.
\end{itemize}
\end{definition}
Conjecture 2.10 (or equivalently 2.11) in \cite{TuncelWolk2005} claims that for all $A, B\in \bS^n$ the weak interlacing property holds. This is equivalent to saying that the eigenvectors of the Jordan-Kronecker product corresponding to the smallest and largest eigenvalues are symmetric.
Conjecture 2.12 in \cite{TuncelWolk2005} claims that for all $A, B\in \bS^n$ the strong interlacing property holds. 
It is worth to note that the interlacing property implies the weak interlacing property. Furthermore, the strong interlacing property is a special case of the interlacing property.

The ``extreme" singular vectors of such expressions have been studied in various contexts as they are related to the numerical stability of methods for solving the algebraic Lyapunov equation. In \cite{ByersNash:1987}, Byers and Nash tried to show that the 
largest singular value of a Lyapunov operator is achieved by a symmetric singular vector.
However, as pointed out in \cite{ChenTian:2015}, the paper contains an error. The proof for $n\leq 5$ has been presented by Chen and Tian in \cite{ChenTian:2015}, whereas the case $n \geq 6$ is still open. The proof of this conjecture on some structured matrices such as non-negative, non-positive, and tridiagonal matrices can be found in \cite{Feng:2015}.
Byers and Nash also conjectured that the singular vector corresponding to the smallest singular value is symmetric \cite{ByersNash:1987} and showed it for $n=2$.
Chen and Tian studied the largest and smallest singular vectors of the \emph{generalized Lyapunov operators}; they showed that the largest singular value is achieved by a symmetric vector for $n\leq 3$, but the smallest singular value is achieved by a symmetric vector for only $n\leq 2$ \cite{ChenTian:2016}.

Although the Jordan-Kronecker product is a special case of the generalized Lyapunov operator, our results as well as our techniques differ from \cite{ChenTian:2016} in many ways. We study the structure of its eigenvectors (not the singular vectors corresponding to the extreme singular values). We show that for $n \leq 3$ the eigenvectors corresponding to the largest as well as the smallest eigenvalues of Jordan-Kronecker product are symmetric. Note that in \cite{ChenTian:2016}, the smallest singular value is achieved by a symmetric vector for $n = 2$, and a counterexample is provided for $n = 3$. 
Differently from \cite{ChenTian:2016}, 
our focus is on the interlacing properties of the eigenvalues corresponding to the symmetric and skew-symmetric eigenvectors of the Jordan-Kronecker product. For interlacing properties, our main results do not only depend on the dimension, $n$, of the matrices but also their rank. See Table~\ref{table:summary} for a summary.

\begin{table}[ht]
\begin{tabular}[l]{l || p{2cm} | p{1.8cm} |p{2cm} }
  \hline			
  \parbox[t]{1.6cm}{ \hspace{1cm}
  \\ ${A, B \in \bS^n \: (\text{or in }\bK^n}$) \\ \hspace{1cm}}& \textbf{\parbox[t]{2cm}{Weak 
  \\Interlacing \\Property}} & \textbf{Interlacing Property} & \textbf{\parbox[t]{2cm}{Strong 
  \\Interlacing \\Property}} \\
  \hline\hline
   $n=2$ & $\checkmark$ \: Cor.~\ref{cor:n2case} & $\checkmark$ \: Cor.~\ref{cor:n2case} & $\checkmark$ Cor.~\ref{cor:n2case}\\
   \hline
   $n=3$ & $\checkmark$ \: Prop.~\ref{prop:n3case} & $\checkmark$ \: Prop.~\ref{prop:n3case} & \xmark \\
   \hline
   \parbox[t]{4.7cm}{$n$ is arbitrary,\\ $\min\{\rank(A), \rank(B)\} \leq 2$} & $\checkmark$ \: Thm.~\ref{thm:rank2case} & $\checkmark$ \: Thm.~\ref{thm:rank2case} & \xmark \\
   \hline
  \parbox[t]{4.7cm}{$n\geq 4$, $\rank(A)=\rank(B)=3$} & \xmark  \: Ex.~\ref{ex:rank33} & \xmark \: Ex.~\ref{ex:rank33} & \xmark \: Ex.~\ref{ex:rank33} \\
  \hline  
  \parbox[t]{4.7cm}{ $k\geq 3$, $n\geq \max\{4,k\}$,\\ $\min\{\rank(A), \rank(B)\} \geq k$} & \xmark \: Prop.~\ref{prop:counterex} & \xmark \: Prop.~\ref{prop:counterex} & \xmark  \: Prop.~\ref{prop:counterex}\\
 \hline
\end{tabular}
\caption{A summary of interlacing properties}
\label{table:summary}
\end{table}

\begin{remark} 
\label{rmk:Bdiag}
Let $A, B \in \bS^n$. Regarding the interlacing properties, without loss of generality, we may assume $B$ is a diagonal matrix with diagonal entries sorted in descending order.
\end{remark}
\begin{proof} Let $A, B \in \bS^n$ be given, and let $B := VDV^\top$ be the spectral decomposition of $B$, where $D$ is the diagonal matrix whose diagonal entries are the eigenvalues of $B$ sorted in descending order. Let $X\in \bS^n$ be such that $\norm{X}_F = 1$.
Define $U:= V^\top X V$ and $\bar{A}:= V^\top A V$. Then
 $U^\top = U$ and $\norm{U}_F^2 = 1$.  Using the commutativity of the trace operator and the orthogonality of $V$, we get
\begin{align*}
\trace \left( AXBX \right) &= \trace \left( \left( V^\top A V\right) \left(V^\top X V\right) \left(V^\top B V\right) \left(V^\top X V\right) \right) = \trace \left( \bar{A}U {D} U\right).
\end{align*}
The proof for $X\in \bK^n$ is similar and is omitted. 
\end{proof}

\subsection{Cases where interlacing properties hold}
We start with some algebraic results related to the interlacing properties. Then, we list some structured matrices for which interlacing properties hold.

\begin{lemma}
\label{lem:eig1}
Let $A\in \bS^n$. Then
\begin{equation*}
\max_{U\in \bS^n, \norm{U}_F=1} \trace(UAU) = \lambda_1 (A)
\hspace{0.2cm} \text{ and } \hspace{0.2cm}
\min_{U\in \bS^n, \norm{U}_F=1} \trace(UAU) =  \lambda_n (A).
\end{equation*}
\end{lemma}
\begin{proof}
We have
\begin{align*}
\max_{\substack{U\in \bS^n,\\ \norm{U}_F=1}} \trace(UAU) = \max_{U\in\bS^n} \dfrac{ \vect (U) ^\top \left (I\otimes A\right)\vect (U)}{\vect (U)^\top \vect (U)}
 \leq \lambda_1 (I\otimes A) 
 = \lambda_1 (A). 
\end{align*}
The first equality follows from $\trace(UAU) = \ovec{U^\top}^\top \ovec{AUI}$ and the identity \eqref{ident:kronvec}.
The inequality holds since the maximization of the Rayleigh quotient of $\left(I\otimes A\right)$ is carried over a subset of $\R^{n^2}$.
If $\hat{U}: = v_1 v_1^\top$, where $v_1$ is an eigenvector of $A$ corresponding to $\lambda_1(A)$, then 
$\trace (\hat{U}A\hat{U}) = \lambda_1(A)$. Since the upper bound of $\max_{{U\in \bS^n, \norm{U}_F=1}} \trace(UAU)$ is achieved with $\hat{U}$, where $\hat{U}\in \bS^n$ with $\norm{\hat{U}}_F= v_1^\top v_1=1$ the first result follows. The proof of the second result is similar and is omitted here.
\end{proof}

Next, we provide a useful result that sheds light on the interlacing relation between the odd and even eigenvalues of $A\otimes B + B\otimes A$.
\begin{proposition}
\label{prop:stronginterlace}
Let $A, B\in \bS^n$ (or $\K^n$). 
Then, the even and odd eigenvalues of $C:=A\otimes B + B\otimes A$ are equal to the eigenvalues of $2(A\overset{s}{\otimes} B)$ and $2(A\overset{\tilde{s}}{\otimes} B)$, respectively. Furthermore, if $G (A\overset{\tilde{s}}{\otimes} B)G^\top$ is a principal submatrix of  $(A\overset{s}{\otimes} B)$ for some $\skewn{n}$-by-$\skewn{n}$ orthogonal matrix $G$, then the odd eigenvalues of $C$ interlace its even eigenvalues.
\end{proposition}
\begin{proof}
Let 
$P:= \begin{bmatrix} Q & \tilde{Q} \end{bmatrix}$. Then  
\begin{align*}
P^\top CP = 
\begin{bmatrix} 
Q^\top C Q & Q^\top C \tilde{Q}\\
\tilde{Q}^\top C Q & \tilde{Q}^\top C \tilde{Q}
\end{bmatrix}
= 
\begin{bmatrix} 
Q^\top C Q & 0\\
0 & \tilde{Q}^\top C \tilde{Q}
\end{bmatrix} 
= \begin{bmatrix} 
2 (A\overset{s}{\otimes} B) & 0\\
0 & 2(A\overset{\tilde{s}}{\otimes} B)
\end{bmatrix}.
\end{align*}
The off diagonal blocks are zero, since 
\[
Q^\top C \tilde{Q} = Q^\top T C T \tilde{Q} = -  Q^\top C \tilde{Q}.
\]
Let $(A\overset{s}{\otimes} B) = U\Lambda_{e}U^\top$ and
$(A\overset{\tilde{s}}{\otimes} B) = V \Lambda_{o}V^\top$
be the spectral decompositions of $(A\overset{s}{\otimes} B)$ and 
$(A\overset{\tilde{s}}{\otimes} B)$, respectively.
Then
\begin{align*}
 C &= \begin{bmatrix} QU & \tilde{Q}V \end{bmatrix} \begin{bmatrix} 2\Lambda_{e} & 0\\ 0 & 2\Lambda_{o} \end{bmatrix}
\begin{bmatrix} QU & \tilde{Q}V \end{bmatrix}^\top.
\end{align*}
By the definition of $Q$ and $\tilde{Q}$, the columns of $QU$ are symmetric and the columns of $QV$ are skew-symmetric. Therefore, the even spectrum of $C$ consists of the eigenvalues of $2(A\overset{s}{\otimes} B)$ and the odd spectrum of $C$ consists of the eigenvalues of $2(A\overset{\tilde{s}}{\otimes} B)$. Since $G$ is orthogonal, the eigenvalues of $G(A\overset{\tilde{s}}{\otimes} B)G^\top$ are the same as those of $(A\overset{\tilde{s}}{\otimes} B)$. Then the result follows by Cauchy's interlacing theorem.
\end{proof}

We remark that a sufficient condition for the interlacing of odd and even eigenvalues of the generalized Jordan-Kronecker product can be stated in the same spirit of Proposition~\ref{prop:stronginterlace}.

In the following, we show that the interlacing property holds for many structured pairs of matrices. Although the weak interlacing property \cite[Conjecture 2.10]{TuncelWolk2005} is stated for real symmetric matrices, we prove that it holds for certain real skew-symmetric matrices as well. 

Real skew-symmetric matrices have purely imaginary eigenvalues; on the other hand, the Jordan-Kronecker product of two skew-symmetric matrices is symmetric and therefore has real eigenvalues. We show that the interlacing property holds for two symmetric matrices as well as of two skew-symmetric matrices, provided one of the matrices has rank at most two.
\begin{thm}
\label{thm:rank2case}
Let $A, B \in \R^{n\times n}$ such that $\min{\left\{\rank(A), \rank(B)\right\}} \leq 2$. 
The odd eigenvalues of $A\otimes B + B\otimes A$ interlace its even eigenvalues when
\begin{enumerate}[(i)]
\item $A$ and $B$ are symmetric matrices, 
\item $A$ and $B$ are skew-symmetric matrices.
\end{enumerate}
\end{thm}
\begin{proof}
$(i)$ Let $A, B \in \bS^n$, we may assume, without loss of generality, that $\rank(B) \leq 2$. So, we let $A := \sum_{i=1}^n\alpha_ia_ia_i^\top$ and $B := \beta_1 e_1e_1^\top +\beta_2 e_2e_2^\top$, where $e_i$ is a vector of all zeros except its $i$th term is $1$. (We used Remark~\ref{rmk:Bdiag}.)
Denote the last $(n-2)$ entries of $a_i$ by $\underline{a}_i := \begin{bmatrix} a_{i3} & a_{i4} & \cdots & a_{in}\end{bmatrix}^\top$ and the $j$th entry of $a_i$ by $a_{ij}$. Let $W:=[w_{ij}] \in \bK^n$ with $\norm{W}_F = 1$. Define $\kw := \ovec{W}$ and $\underline{w}_i := \begin{bmatrix} w_{3i} & w_{4i} & \cdots & w_{ni}\end{bmatrix}^\top$, which consists of the last $(n-2)$ entries of $W(:, i)$, where $i\in\{1,2\}$. Then
\begin{align*}
\kw^\top (B\otimes A)\kw &= \beta_1\left( \alpha_1\left(W(:,1)^\top a_1\right)^2 + \cdots+\alpha_n\left(W(:,1)^\top a_n\right)^2\right)\\
& +    \beta_2\left( \alpha_1\left(W(:,2)^\top a_1\right)^2 +\cdots+ \alpha_n\left(W(:,2)^\top a_n\right)^2\right) \\
&= \beta_1\sum_{i=1}^n\alpha_i\left(w_{21}^2a_{i2}^2+ (\underline{w}_1^\top\underline{a}_i)^2+ 
2 w_{21}a_{i2}\left(\underline{w}_1^\top\underline{a}_i\right)\right)\\
&+ \beta_2\sum_{i=1}^n\alpha_i\left(w_{21}^2a_{i1}^2+ (\underline{w}_2^\top\underline{a}_i)^2
-2 w_{21}a_{i1}\left(\underline{w}_2^\top\underline{a}_i\right)\right).
\end{align*}
Let $U:= [u_{ij}] \in {\bS}^n$ be such that  all of its diagonal elements are zero.
Define $\ku := \ovec{U}$ and $\underline{u}_i := \begin{bmatrix} u_{3i} & u_{4i} & \cdots & u_{ni}\end{bmatrix}^\top$, for $i\in\{1,2\}$. 
Choosing $u_{21}:=-w_{21}$, $\underline{u}_1:=-\underline{w}_1$ and assigning the upper triangular part of $W(2:n,2:n)$ to the upper triangular part of ${U}(2:n,2:n)$ gives $\norm{U}_F =\norm{W}_F =1$. Then
\begin{align*}
\ku^\top (B\otimes A)\ku &= \beta_1\left( \alpha_1\left(U(:,1)^\top a_1\right)^2  + \cdots+\alpha_n\left(U(:,1)^\top a_n\right)^2\right)\\
& +    \beta_2\left( \alpha_1\left(U(:,2)^\top a_1\right)^2 +  \cdots + \alpha_n\left(U(:,2)^\top a_n\right)^2\right) \\
&= \beta_1\sum_{i=1}^n\alpha_i\left(u_{21}^2a_{i2}^2+ (\underline{u}_1^\top\underline{a}_i)^2+ 
2 u_{21}a_{i2}\left(\underline{u}_1^\top\underline{a}_i\right)\right)\\
&+ \beta_2\sum_{i=1}^n\alpha_i\left(u_{21}^2a_{i1}^2+ (\underline{u}_2^\top\underline{a}_i)^2
+2 u_{21}a_{i1}\left(\underline{u}_2^\top \underline{a}_i\right)\right)\\
& = \kw^\top (B\otimes A)\kw.
\end{align*}
This proves that for a given $W\in \bK^n$ with $\norm{ W}_F = 1$, there exists $U\in \bS^n$ with $\norm{U}_F =1$ such that
$\trace (AWBW^\top ) = \trace (AUBU )$.
Therefore, the weak interlacing property holds. 

Define the diagonal matrix $\Phi \in \bS^{\skewn{n}}$ by
\begin{align*}
\Phi_{kk} := \begin{cases}
    -1, & \text{if } k \in\{1,2, \ldots, n-1\},\\
    1, & \text{ otherwise}.
  \end{cases}
\end{align*}
$\Phi (A\overset{\tilde{s}}{\otimes}B) \Phi$ is a principal submatrix of $A\overset{s}{\otimes}B$. Then by Cauchy's interlacing theorem and Proposition~\ref{prop:stronginterlace}, it follows that the odd eigenvalues of $\left(A\otimes B + B\otimes A\right)$ interlace its even eigenvalues. Thus, the interlacing property holds. \\
$(ii)$ The proof for the skew-symmetric matrices is similar and is omitted. We refer the interested reader to the thesis \cite{thesis:nkalantarova} for the proof of this part.
\end{proof}
\begin{corollary}
\label{cor:n2case}
Let $A, B \in \bS^2$ (or $\bK^2$). Then the weak, interlacing and strong interlacing properties hold for $A, B$.
\end{corollary}

\begin{proposition}
\label{prop:n3case}
Let $A, B\in \bS^3$. Then the weak interlacing and the interlacing properties hold for $A, B$.
\end{proposition}
\begin{proof}
Let $W\in \bK^3$ be given such that $\|{W}\|_F=1$. We will show that there exists $U\in \bS^3$ such that
$\|{U}\|_F=1$ and $\trace(AWBW^\top) = \trace(AUBU)$.

Since $W\in \bK^3$, $W$ is at most of rank $2$. (Recall that real skew-symmetric matrices have purely imaginary eigenvalues, and these eigenvalues come in pairs, $\pm i\lambda$.)
Then 
\begin{align*}
W := Q\begin{bmatrix}0 & \lambda_1 & 0\\
-\lambda_1 &0 &0\\
0&0&0 \end{bmatrix}Q^\top,
\end{align*}
for some orthogonal matrix $Q$ such that $2\lambda_1^2=1$.
Then
\[
\trace(AWBW^\top) = \trace(AQ\tilde{\Lambda}Q^\top BQ\tilde{\Lambda}^\top Q^\top) = \trace(\bar{A}\tilde{\Lambda}\bar{B}\tilde{\Lambda}^\top)=\trace(\bar{A}_0\tilde{\Lambda} \bar{B}_0\tilde{\Lambda}^\top),\]
where $\bar{A}_0, \bar{B}_0$ is the $2$-by-$2$ leading principal matrix of $\bar{A}$ and $\bar{B}$, respectively.  
Since, both $\bar{A}_0, \bar{B}_0$ are of rank at most $2$, the rest follows from Theorem~\ref{thm:rank2case}. 
\end{proof}

Consider simultaneously diagonalizable matrices. Some examples of this class are diagonal matrices and circulant matrices, which are used extensively in signal processing and statistics. A result from \cite{TuncelWolk2005} shows that the strong interlacing (which implies the weak interlacing) property holds for pair of simultaneously diagonalizable matrices. We restate it below.

\begin{lemma}\emph{\cite[Corollary 2.5]{TuncelWolk2005}}
\label{lem:commuting}
Let $A, B \in \bS^n$ be commuting matrices. Let
$\lambda_i, \mu_j$ denote the eigenvalues and $v_i, v_j$ be the corresponding eigenvectors of
$A$ and $B$, respectively. Then, for $1\leq i \leq j \leq n$, we get $\frac{1}{2}(\lambda_i \mu_j + \lambda_j \mu_i)$ as the 
eigenvalues and $\svec{v_i v_j^\top + v_j v_i^\top}$ as the corresponding eigenvectors of $A\overset{s}{\otimes} B$, where
$\svect : \bS^n\rightarrow \R^{\symn{n}}$ is a mapping such that such that for every $n$-by-$n$ real symmetric matrix $X:=[x_{ij}]$,
\begin{equation*}
\small
\svec{X} := \begin{bmatrix} x_{11}&  \sqrt{2}x_{21} & \cdots &  \sqrt{2}x_{n1} & x_{22}&  \sqrt{2}x_{32} &\cdots &  \sqrt{2}x_{n2} & \cdots & x_{nn}\end{bmatrix}^\top.
\end{equation*}
\end{lemma} 
Lemma~\ref{lem:commuting} also applies to a symmetrized similarity operator \cite{Zhang:1998} as follows. For every nonsingular matrix $P \in \R^{n\times n}$, \cite{Zhang:1998} define $H_P:\R^{n\times n} \rightarrow \bS^n$ by
\[
H_P(X):= PXP^{-1} + P^{-\top} X^\top P^\top.
\]
Let us restrict the domain of $H_P$ to $\bS^n$ and restrict $P$ to symmetric matrices. Then the resulting operator $H_P$ is representable by a Jordan-Kronecker product of symmetric matrices $P$ and $P^{-1}$. Since $P, P^{-1}$ commute, Lemma~\ref{lem:commuting} applies.

Given the result for simultaneously diagonalizable matrices, an interesting direction to explore is to determine how one can perturb $A$ or $B$ so that the weak, the strong, or the interlacing property will still be preserved. 
The next two propositions can be proved using the proof technique given for Theorem~\ref{thm:rank2case}. So, proofs of Propositions~\ref{prop:diagpert} and \ref{prop:diagpert2} are omitted.
\begin{proposition}
\label{prop:diagpert}
Let $A, B \in \bS^n$ be diagonal matrices, where $A_{kk} := a_{kk}$, for every $k\in\{1, 2, \ldots, n\}$. Let 
\[\tilde{A} := A +  \sum_{i=1}^{ \floor*{k/2}}\alpha_i^{(k)}(E_{i (k-i+1)} + E_{i (k-i+1)}^\top), \: \text{ i.e.,} \]

\begin{equation*}
\small
\tilde{A}: = 
\begin{bmatrix}
 a_{11}    		& 0        		& \cdots             &  \cdots   &     \cdots         &  0                     &  \alpha_1^{(k)}    & 0      	& \cdots & 0\\ 
 0        		& a_{22}    		& 0            &  \cdots   & 0           &\alpha_2^{(k)}    &  0                        &  \vdots	        & \cdots & 0\\ 
\vdots   		& 0       		& \ddots    &      & \iddots   &  0      	            &  	\vdots	                 &       	& \cdots & 0\\ 
\vdots   		& \vdots      	& \vdots    &  \iddots  & \iddots      &  \vdots      	    & \vdots 		        & \vdots 	&  & \vdots \\  
\vdots    		& 0        		&  \iddots  &  \iddots     & \ddots   &  0      		    &   \vdots		&  	&  & \\
0         		& \alpha_2^{(k)}& 0           &  \cdots     & 0          &  a_{(k-1)(k-1)}             &  0 			& \vdots 	&  & \vdots\\ 
\alpha_1^{(k)}   & 0        		& \vdots     &       & \vdots    &  0      		    &  a_{kk} 			& 0 		&  & \\ 
0    			& \vdots            & \vdots     &       & \vdots    &  \vdots                 & 0 			& \ddots 	& \ddots & \vdots\\ 
\vdots    		&  \vdots       	& \vdots     &       & \vdots  &   \vdots     		            & \vdots 		& \ddots 	& \ddots & 0\\ 
0    			& 0        		& 0      &  \cdots     & 0        &  0      		    & 0 			& \cdots 	& 0      & a_{nn}\\   
\end{bmatrix},
\end{equation*}
where $E_{ik}$ is an $n$-by-$n$ matrix with all entries zero except $E(i,k) = 1$, and $\alpha_i^{(k)} \in \R$, $i\in\{1,2, \ldots,\floor*{k/2}\}.$
Then, the interlacing property (as well as the weak interlacing property) holds for the pair $\tilde{A}, {B}$. 
\end{proposition}
\begin{proposition}
\label{prop:diagpert2}
Let $A$ be an $n$-by-$n$ symmetric tridiagonal matrix and $B$ be an $n$-by-$n$ diagonal matrix. Then the odd eigenvalues of $A\otimes B+ B\otimes A$ interlace its even eigenvalues. 
Furthermore, define $\bar{A}\in \bS^n$ such that   $\bar{A}_{ij} = A_{ij}$ for all $i, j \in \{1,\ldots, n\}$ except $A_{rs} \not = 0$, $(r>s+1)$, and $A_{(s+1)s}=0$ for some $r, s$. Then  the odd eigenvalues of $\bar{A}\otimes B+ B\otimes \bar{A}$ interlace its even eigenvalues. 
\end{proposition}

Proposition~\ref{prop:diagpert2} gives another sufficient condition for the weak interlacing property. 
For every pair of matrices $A, B \in \bS^n$ where $B$ is diagonal and $A$ is formed by perturbing a diagonal matrix where each row of the upper triangular matrix of $A$ has at most one nonzero entry, the odd eigenvalues of $A\otimes B+ B\otimes A$ interlace its even eigenvalues.

The eigenvalues of  the Jordan-Kronecker product that correspond to the symmetric eigenvectors are the same as the eigenvalues of $2(A\overset{s}{\otimes} B)$. In Theorem 2.9 \cite{TuncelWolk2005}, the authors proved that \; $\forall A, B\in\bS^n$, \;
$A\otimes B \succ 0  \;\;\; \Longleftrightarrow  \;\;\; A\overset{s}{\otimes} B\succ 0.$
From this, one can easily obtain
\[
(A\overset{s}{\otimes} B)\succ  0  \;\;\; \Longleftrightarrow  \;\;\;  A{\otimes} B+ B\otimes A\succ 0.
\]
However, $(A\overset{\tilde{s}}{\otimes} B)\succ  0$ does not necessarily imply $(A{\otimes} B)\succ 0$ (there exist numerous counterexamples). 
If either $A$ (or $B$) has a zero eigenvalue corresponding to an eigenvector, say, $v$, then taking $U:=vv^\top$ shows that the eigenvector of $A\otimes B+ B\otimes A$ corresponding to the smallest eigenvalue is symmetric, as 
\[
\trace(AUBU)= \trace(Avv^\top Bvv^\top) = \trace\left((v^\top Av)v^\top Bv\right) = 0.
\]
Furthermore, based on extensive numerical experiments, it seems that the minimum eigenvalue of the Jordan-Kronecker product of two positive definite matrices corresponds to a symmetric eigenvector. We state this as a conjecture.
\begin{conjecture} 
Let $A, B \in \bS^n$ be positive definite matrices. Then
\begin{equation}
\label{conj_new_1}
\min_{Tu = u}\dfrac{u^\top (A \otimes B) u}{u^\top u} \leq \min_{Tw = -w}\dfrac{w^\top (A \otimes B) w}{w^\top w}.
\end{equation}
\end{conjecture}
In primal-dual interior-point methods for semidefinite optimization, the new iterates are computed by adding the search directions to the current primal and dual points using a line search to make sure the new iterates are feasible. 
Usually the feasibility condition requires checking the positive semidefiniteness of certain matrices, which is determined by the minimum eigenvalue of those matrices. In this respect, we believe the study of this conjecture may be useful.

The following theorem provides a set of perturbations which allows constructing nontrivial pairs of matrices satisfying the weak interlacing property. Furthermore, it helps improve our understanding of the spectral properties of $A\otimes B+ B\otimes A$ in terms of the spectral properties of $A$ and $B$.
\begin{thm}
\label{thm:perturbs}
Let $A, B \in \bS^n$. Then for $\mu>0$ large enough, the weak interlacing property holds for
\begin{enumerate}[(i)]
\item $(A+\mu I, B)$, if $B$ is indefinite and the multiplicity of the smallest and the largest eigenvalues of $B$ is $1$.
\item $(A+\beta B, B+\mu A)$ where $\beta\in \R$, if $A+\beta B$ is indefinite and the multiplicity of the smallest and the largest eigenvalues of $A+\beta B$ is $1$.
\item $(A+\mu D, B)$ where $B$ and $D$ are diagonal, if $B\otimes D$ is indefinite and the multiplicity of the smallest and the largest eigenvalues of $B\otimes D$ is $1$.
\item $(A+\mu D_1, B+\mu D_2)$, if 
$D_1\otimes D_2$ is indefinite and the multiplicity of the smallest and the largest eigenvalues of $D_1\otimes D_2$ is $1$.
\end{enumerate}
\end{thm}
\begin{proof}
$(i)$ Let $B$ be an indefinite symmetric matrix and $\mu>0$. 
Since
\begin{align*}
\trace \left((A+\mu I)U BU\right) &= \trace (AUBU) +\mu \trace (UBU)\\
\trace \left((A+\mu I)W BW^\top\right) &= \trace (AWBW^\top) +\mu \trace \left(WBW^\top\right)
\end{align*}
By Lemma~\ref{lem:eig1}, $\lambda_1 (B) =\max_{U\in \bS^n, \norm{U}_F=1} \trace(UBU)$, 
$\lambda_n (B) = \min_{U\in \bS^n, \norm{U}_F=1} \trace(UBU)$. 
If the eigenspaces of the largest and smallest eigenvalues of $B$ both have dimension $1$, then 
\begin{align*}
\max_{U\in \bS^n, \norm{U}_F=1} &\trace(UBU) > \max_{W\in \bK^n, \norm{W}_F=1} \trace\left(WBW^\top\right), \; \text{ and }\\
\min_{U\in \bS^n, \norm{U}_F=1} &\trace(UBU) < \min_{W\in \bK^n, \norm{W}_F=1} \trace\left(WBW^\top\right).
\end{align*}
Then, for $\mu$ large enough
\begin{align*}
\max_{U\in \bS^n, \norm{U}_F=1} &\trace \left(\left(A+\mu I\right)U BU\right)  \geq \max_{W\in \bK^n, \norm{W}_F=1} \trace \left(\left(A+\mu I\right)W BW^\top\right), \; \text{ and }\\
\min_{U\in \bS^n, \norm{U}_F=1} &\trace \left(\left(A+\mu I\right)U BU\right)  \leq \min_{W\in \bK^n, \norm{W}_F=1} \trace \left(\left(A+\mu I\right)W BW^\top\right).
\end{align*}
The proofs for parts $(ii)- (iv)$ are along similar lines and are omitted.
\end{proof}

\subsection{Cases when interlacing properties fail}
In Theorem~\ref{thm:rank2case}, we showed that the interlacing property holds for the pair of matrices $A, B \in \bS^n$ (or $\bK^n$), whenever $$\min\{ \rank(A), \rank(B)\}\leq 2.$$ 
\noindent Given this result, one may ask whether the interlacing property holds for pairs of matrices $A, B$ such that $\min\{ \rank(A), \rank(B)\}= 3$.
We answer this question below in Proposition~\ref{prop:counterex} in a more general setting.
The following example shows that for some $A, B\in \bS^4$ with $\rank(A) = \rank(B)=3$, the weak interlacing property fails. This example is significant in the sense that the validity of the conjecture depends on the rank as well as the dimension of the pair. In particular, it helps to prove that even though the interlacing property holds for every pair of $A, B \in \bS^3$, it may fail when $n\geq 4$ and $\min\{ \rank(A), \rank(B)\}\geq 3$. 
\begin{example}
\label{ex:rank33}
Consider the following pair 
\begin{align*}
A_0:=
\begin{bmatrix}      
  2 & -3  &  2 &  -3\\
 -3 &   0 &   1 &   2\\
  2 &   1 &  -2 &   1\\
 -3 &   2 &   1 &  -4
\end{bmatrix},\;\; B_0:=\Diag(1,    2,    -2,     0).      
\end{align*}
Here, $\rank(A_0) = \rank(B_0)=3$. However, the eigenvector corresponding to the smallest eigenvalue of $A_0\otimes B_0+ B_0\otimes A_0$ is skew-symmetric (see Appendix A for a proof). Therefore, the weak interlacing property fails, and so do the interlacing as well as the strong interlacing properties.
\end{example}

\begin{proposition}
\label{prop:counterex}
For every pair of integers $m\geq k\geq 3$ and for every integer $n$ with ${n\geq \max\{4, k\}}$, there exist $n$-by-$n$ symmetric matrices $A, B$ with $\min\{\rank(A),\rank(B)\}=k$ and 
$\max\{\rank(A),\rank(B)\}= m$ for which the weak interlacing property fails.
\end{proposition}
\begin{proof}
If for some unit-norm matrix ${W}_0\in\bK^n$, i.e., $\norm{W_0}_F=1$,  $$\trace (AW_0BW_0^\top) < \trace (AUBU) \hspace{0.6cm} (\text{or } \trace (AW_0BW_0^\top)> \trace (AUBU))  $$ 
holds for every symmetric $n$-by-$n$ unit-norm matrix $U$, then the smallest (or the largest) eigenvalue of $(A\otimes B+B\otimes A)$ cannot be a symmetric matrix. This shows that the weak interlacing property fails for $A, B$.

Using the pair $A_0, B_0$ from Example~\ref{ex:rank33} where $k = m = 3$, $n = 4$, and
\begin{equation*}
\small
W_0:= 
\begin{bmatrix} 
 0 &  2 & -4 &  0\\
 -2  &  0 &   4 &  -3\\
 4 &  -4 &  0 &  -2\\
 0 &   3 &   2 &   0
\end{bmatrix},
\end{equation*}
one can show that the minimum eigenvalue of $A_0\otimes B_0 + B_0\otimes A_0$ corresponds to a skew-symmetric eigenvector. Hence this pair provides an instance where the weak interlacing property fails for $k = 3 , m = 3, n= 4$, i.e., $\rank(A_0)=\rank(B_0)= 3$ and $A_0, B_0 \in \bS^4$.

By choosing $A_1:=A_0$, $B_1:=B_0+\varepsilon I$, where $\varepsilon>0$ is small enough so that $\rank(B) = 4$, one can construct an example for which $k=3, m=4, n=4$ and  the weak interlacing property fails. 
To show that the weak interlacing property fails for $k=4, m =4, n=4$, choose $A_1:=A_0+\varepsilon I$, $B_1:=B_0+\varepsilon I$, where $\varepsilon>0$ is small enough so that $\rank(A_1) = \rank(B_1) = 4$. For $k=4, m \geq 4$ and for $n>4$, it suffices to choose 
\[A_2 := \begin{bmatrix} A_1 & 0\\ 0 & \varepsilon I_{d}\end{bmatrix} \in \bS^n \text{ and } B_2 := \begin{bmatrix} B_1 & 0\\ 0 & 0\end{bmatrix} \in \bS^n,\] 
where $\varepsilon >0$ a small number and $I_d$ is a (partial) identity matrix with the first $d \in\{0,1,\ldots, m-k\}$ entries of its diagonal equal to $1$ and the rest zero.
The proof for arbitrary $k$ follows along similar lines with the case $n=k=4$ and is omitted here. For a more detailed proof please refer to \cite{thesis:nkalantarova}.
\end{proof}

For nonnegative symmetric matrices $A, B\in\bS^n$, \eqref{conj1:eq2} holds by the Perron-Frobenius theorem, i.e., for nonnegative symmetric matrices the largest eigenvalue corresponds to a symmetric eigenvector; however, we constructed examples of such nonnegative symmetric matrices where \eqref{conj1:eq1} fails. In addition, we constructed examples where for a pair of full rank $6$-by-$6$ real skew-symmetric matrices the weak interlacing property fails. See Appendix A, Example~\ref{counter:nonneg_skew}.

Using a similar construction given as in Theorem~\ref{thm:perturbs}, we generate infinitely many pairs of matrix pencils for which  the weak interlacing property fails.
\begin{thm}
\label{thm:perturbs2}
If $\bar{A}, \bar{B} \in \bS^n$ where $n\geq 4$ and the weak interlacing property fails, then for every $A, B \in \bS^n$, the weak interlacing fails for the following pairs:
\begin{enumerate}
\item $(A+\mu \bar{A}, \bar{B})$ and $(A+\beta\mu \bar{A}, B+\mu \bar{B})$ for $\mu>0$ large enough and $\beta>0$, 
\item $(A+\beta \bar{B}, B+\alpha \bar{A})$ for $\alpha, \beta >0$ large enough. 
\end{enumerate}
\end{thm}

\noindent Proof of Theorem~\ref{thm:perturbs2} is elementary (similar to the proof of Theorem~\ref{thm:perturbs}) and is omitted.

\section{Conclusion}
We proved that for Jordan-Kronecker products of $n$-by-$n$ symmetric matrices the odd eigenvalues interlace the even eigenvalues provided one of the matrices has rank at most two. Under the same condition, this property holds for Jordan-Kronecker products of skew-symmetric matrices as well. Moreover, we showed that the weak interlacing and the interlacing properties hold for every pair of  $n$-by-$n$ symmetric (and skew-symmetric) matrices, whenever $n\in\{2,3\}$. In many applications of the Kronecker product, one is interested in constraints where one of the terms in the Kronecker product has rank one or two. Our positive results may be helpful in such applications.
Since we also proved that the weak and hence the strong interlacing properties generally fail, this intensifies the motivation for finding specially structured matrices for which these interlacing properties hold.

In the introduction, we exposed some nice features of eigenspaces of Jordan-Kronecker products of pairs of symmetric and skew-symmetric matrices. One may also wonder if similar characterizations can be established for matrices of the form $A\otimes B - B\otimes A$. For $A, B \in \R^{n\times n}$, we define $A\otimes B - A\otimes B$ as the \emph{Lie-Kronecker product} of $A$ and $B$. 
Note that for every pair of symmetric matrices $A, B$ (or skew-symmetric matrices $A, B$), $A\otimes B$ is symmetric.

The following proposition characterizes the eigenvector/eigenvalue structure of the Lie-Kronecker product of symmetric and skew-symmetric matrices.
\begin{proposition}
Let $A, B\in \bS^{n}$ (or $\bK^n$). Then the following hold.
\begin{enumerate}
\item If $\lambda\not = 0$ is an eigenvalue of $A\otimes B - B\otimes A$ corresponding to the eigenvector $v$, then
$-\lambda$ is also an eigenvalue corresponding to the eigenvector $Tv$.
\item The eigenvectors of $A\otimes B - B\otimes A$ can be chosen in the following form 
\[\{v_1, v_2, \ldots, v_t, Tv_1, Tv_2, \ldots, Tv_t, u_1,u_2, \ldots, u_n\},\]
where $t := n(n-1)/2$ and $Tu_i = u_i$ for every $i\in\{1,2,\ldots, n\}$.
Furthermore, the symmetric eigenvectors $\{u_1, u_2, \ldots, u_n\}$ belong to the null-space of $A\otimes B - B\otimes A$.
\end{enumerate}
\end{proposition}

\section{Acknowledgements}
This work is based on research supported in part by Discovery Grants from NSERC and by U.S. Office of Naval Research under award numbers N00014-15-1-2171 and N00014-18-1-2078. Part of this work was done while the second author was visiting Simons Institute for the Theory of Computing, supported in part by the DIMACS/Simons Collaboration on Bridging Continuous and Discrete Optimization through NSF grant \#CCF-1740425.

\section{References}
\bibliography{article_u1}
\newpage
\appendix
\section{}
\begin{proposition}
Consider the following $4$-by-$4$ symmetric matrices:
\begin{equation}
\label{counterex:0}
A_0 :=
\begin{bmatrix}      
  2 & -3  &  2 &  -3\\
 -3 &   0 &   1 &   2\\
  2 &   1 &  -2 &   1\\
 -3 &   2 &   1 &  -4
\end{bmatrix},\;\;     
 \text{ and }
B_0 := \begin{bmatrix}
1 & 0 & 0 & 0\\
0 & 2 & 0 & 0\\
0 & 0 & -2 & 0\\
0 & 0 & 0 & 0   
\end{bmatrix}.
\end{equation}
The smallest eigenvalue of $A_0 \otimes B_0 + B_0 \otimes A_0$ corresponds to a skew symmetric eigenvector. 
\end{proposition}
\begin{proof}
Here, $\rank(A_0)=\rank(B_0)=3$. For the following skew-symmetric matrix 
\[{W}_0:= 
\begin{bmatrix} 
 0 &  2 & -4 &  0\\
 -2  &  0 &   4 &  -3\\
 4 &  -4 & 0 &  -2\\
 0 &   3 &   2 &   0
\end{bmatrix},\] 
 define ${\kw}_0:=\ovec{W_0}$. Then
\[
\rho_{\kw_0}(A_0, B_0):=\dfrac{{\kw_0}^\top (A_0 \otimes B_0){\kw_0}}{{\kw_0}^\top {\kw_0}} = \dfrac{-1472}{98} < -15.
\] 
If we show that for every symmetric vector $\ku$,
 \[ \Delta_{\ku} := \ku^\top (A_0 \otimes B_0) \ku +15 \ku^\top \ku >0,\] 
then that will imply the minimum eigenvalue of $A_0\otimes B_0 + B_0\otimes A_0$ corresponds to a skew-symmetric vector.
Note that we can get a lower dimensional quadratic representation of $\Delta_{\ku}$ by gathering the 
terms for each distinct entry $u_{ij}$.
\begin{align*}
\Delta_{\ku} &=  \ku^\top (A_0 \otimes B_0) \ku +15 \ku^\top \ku\\
 &= \begin{bmatrix} u_{11}\\  u_{21}\\ u_{31}\\ u_{41}\\ u_{22}\\ u_{32}\\ u_{42}\\ u_{33} \\u_{43}\\ u_{44} \end{bmatrix} ^\top
\begin{bmatrix}
  17 & -3 & 2 & -3 & 0 & 0 & 0 & 0 & 0 & 0\\
-3 & 34 & 1 & 2 & -6 & 4 & -6 & 0 & 0 & 0\\
2 & 1 & 24 & 1 & 0 & 6 & 0 & -4 & 6 & 0\\
-3 & 2 & 1 & 26 & 0 & 0 & 0 & 0 & 0 & 0\\
0 & -6 & 0 & 0 & 15 & 2 & 4 & 0 & 0 & 0\\
0 & 4 & 6 & 0 & 2 & 26 & 2 & -2 & -4 & 0\\
0 & -6 & 0 & 0 & 4 & 2 & 22 & 0 & 0 & 0\\
0 & 0 & -4 & 0 & 0 & -2 & 0 & 19 & -2 & 0\\
0 & 0 & 6 & 0 & 0 & -4 & 0 & -2 & 38 & 0\\
0 & 0 & 0 & 0 & 0 & 0 & 0 & 0 & 0 & 15  
\end{bmatrix} 
\begin{bmatrix} u_{11}\\  u_{21}\\ u_{31}\\ u_{41}\\ u_{22}\\ u_{32}\\ u_{42}\\ u_{33}\\ u_{43}\\ u_{44} \end{bmatrix}.
\end{align*}

Since the last row and column of the above matrix are all zeros except its diagonal element which is positive, it suffices to show that the leading $9$-by-$9$ matrix, which we denote by $T$, is positive definite.
Furthermore, $5$-by-$5$ leading principal matrix of $T$ given below
\[X :=\begin{bmatrix}
  17 & -3 & 2 & -3 & 0 \\
-3 & 34 & 1 & 2 & -6  \\
2 & 1 & 24 & 1 & 0 \\
-3 & 2 & 1 & 26 & 0 \\
 0 & -6 & 0 & 0 & 15
\end{bmatrix}
\]
is positive definite, since it is a symmetric strictly diagonally-dominant matrix.
In order to show that $T$ is positive definite, we use Schur Complement Lemma (see, for instance, \cite{book:Tuncel}). Let 
$Z:=T(6:9,6:9)$ and $Y:=T(1:5,6:9)$ (in \textsf{MATLAB} notation). 
We compute $
U:=Z-Y^\top X^{-1} Y$ and show $U\succ 0$.
\[
U = 
1189419 \begin{bmatrix}
28029302 & 2502346 & -1211070 & -6509328\\
2502346 & 24153701 & 36432 & -54648\\
-1211070 & 36432 & 21793953 & -1171326\\
-6509328 & -54648 & -1171326 & 43386654
\end{bmatrix}
\]
$U$ is positive definite, since each of its diagonal elements is positive, it is symmetric and strictly diagonally-dominant.
This concludes the proof.
\end{proof}
\begin{example}
\label{counter:nonneg_skew}
For the following pair of $4$-by-$4$ nonnegative symmetric matrices and $6$-by-$6$ skew-symmetric matrices, the weak interlacing property fails: 
 \begin{equation*}
 A_+ := 
\begin{bmatrix}
	98 &  48 &  88 &  31\\
    48 &  33 &  91 & 116\\
    88 &  91 &  91 &  45\\
    31 & 116 &  45 & 139
\end{bmatrix} \text{, }  
B_+: = 
\begin{bmatrix} 35 &  23 &  78 & 125\\
        23 & 100 &  91 & 152\\
        78 &  91 &   1 & 120\\
       125 & 152 & 120 & 187     
\end{bmatrix}.
\end{equation*} 
\begin{equation*}
\tilde{A}:=
\begin{bmatrix}
0 & -20 &   -9  &   4  & -14  &  -13\\
 20  &    0  &  33  &  -18  &   0  & -20\\
   9  & -33  &   0  &   2  &  -7  &  -8\\
  -4  &  18  &  -2  &   0  & -13  & -10\\
  14  &   0  &   7  &  13  &   0  & -22\\
  13  &  20  &   8  &  10  &  22  &   0
\end{bmatrix} \text{, }
\end{equation*}
\begin{equation*}
\tilde{B}:= 
\begin{bmatrix}
	 0  & -8  &  2 & -13 & -20 & 28\\
  8   &    0   &   0   & -17   &  16   &  28\\
  -2   &   0   &   0   &   9   &   6   &  -37\\
  13   &  17   &  -9   &   0   &  14   &  -4\\
  20   & -16   &  -6   & -14   &   0   &  -4\\
 -28   &  -28   &   37   &   4   &   4   &   0
\end{bmatrix}.
\end{equation*}
For each pair above, the eigenvector corresponding to the smallest eigenvalue of the Jordan-Kronecker product (with multiplicity one) is a skew-symmetric vector.
\end{example}

\end{document}